\documentclass[12pt]{amsart}

\usepackage{amsmath, amscd, graphicx, latexsym, hyperref, rlepsf, times, }

\input diagrams

\oddsidemargin .25in \theoremstyle{plain}

\newtheorem{Thm}{Theorem}

\newtheorem{Lem}[Thm]{Lemma}

\newtheorem{Cor}[Thm]{Corollary}

\newtheorem{Def}[Thm]{Definition}

\newtheorem{Rem}[Thm]{Remark}

\newcommand{\hgt}{\operatorname{ht}}

\newcommand{\bfz}{{\mathbb{Z}}}

\newcommand{\bfq}{{\mathbb{Q}}}

\newcommand{\OB}{\mathcal{OB}}

\def\a{\alpha}
\def\b{\beta}
\def\d{\delta}
\def\g{\gamma}

\def\s{\sigma}

\begin{document}

\title{Symplectic fillings of lens spaces as Lefschetz fibrations}

\author{Mohan Bhupal and Burak Ozbagci}


\address{Department of Mathematics,  METU, Ankara, Turkey, \newline bhupal@metu.edu.tr}

\address{Department of Mathematics, Ko\c{c} University, Istanbul, Turkey
\newline bozbagci@ku.edu.tr}


\begin{abstract}

We construct a positive allowable Lefschetz fibration over the disk
on any minimal (weak) symplectic filling of the canonical contact structure
on a lens space. Using this construction we prove that any minimal
symplectic filling of the canonical contact structure on a lens
space is obtained by a sequence of rational blowdowns from the
minimal resolution of the corresponding complex two-dimensional
cyclic quotient singularity.

\end{abstract}

\maketitle

\section{Introduction}

The link of an isolated complex surface singularity carries a canonical---also known as
Milnor fillable---contact structure which is unique up to isomorphism \cite{cnp}.  A
Milnor fillable contact structure is Stein fillable since a regular neighborhood of the
exceptional divisor in a minimal resolution of the surface singularity provides a
holomorphic filling which can be deformed to be Stein without changing the contact
structure on the boundary \cite{bd}. In particular, a singularity link with its
canonical contact structure always admits a symplectic filling given by the
minimal resolution of the singularity.

The canonical contact structure on a lens space (the oriented link
of a complex two-dimensional cyclic quotient singularity) is well
understood as the quotient of the standard tight contact structure
on $S^3$. The finitely many diffeomorphism types of the minimal
symplectic fillings of the canonical contact structure on a lens
space were classified by Lisca \cite{l} (see also work of the first
author and K. Ono \cite{bono}).

In this paper, we give an algorithm to present each minimal
symplectic filling of the canonical contact structure on a lens
space as an explicit genus-zero PALF (positive allowable Lefschetz
fibration) over the disk. The existence of such a genus-zero PALF
also follows from \cite[Theorem 1]{w} although we do not rely on
that result in this paper.

Using our construction we prove that any minimal symplectic filling
of the canonical contact structure on a lens space is obtained by a
sequence of rational blowdowns (cf. \cite{fs}, \cite{p}) along
\emph{linear} plumbing graphs starting from the minimal resolution
of the corresponding cyclic quotient singularity. As a corollary, we
show that the canonical contact structure on a lens space admits a
unique minimal symplectic filling---represented by the Stein
structure via the PALF we construct on the minimal resolution---up
to symplectic rational blowdown and symplectic deformation
equivalence.

We refer the reader to \cite{gs} and \cite{ozst} for background
material on Lefschetz fibrations, open books and contact structures.
We denote a right-handed Dehn twist along a curve $\g$ as $\g$ again
and we use functional notation while writing products of Dehn
twists.

\section{Symplectic fillings as Lefschetz fibrations}\label{lef}

For integers $1 \leq q < p$, with $(p,q)=1$, recall that the
Hirzebruch-Jung continued fraction is given by
$$ \frac {p}{q} = [a_1,a_2,\ldots,a_l] = a_1-
\cfrac{1}{a_2- \cfrac{1}{\ddots- \cfrac{1}{a_l}}}, \qquad a_i\geq 2
\text{ for all $1\leq i\leq l$}. $$ The lens space $L(p,q)$ is
orientation preserving diffeomorphic to the link of the cyclic
quotient singularity whose minimal resolution is given by a linear
plumbing graph with vertices having weights $-a_1,-a_2,\ldots,
-a_l$, where $p/q=[a_1, \ldots, a_l]$.

It is known that any tight contact structure on $L(p,q)$, in
particular the canonical contact structure $\xi_{can}$,   is
supported by a planar open book \cite{sc}. According to Wendl
\cite{w}, if a contact $3$-manifold $(Y, \xi)$ is supported by a
planar open book $\OB_\xi$, then any strong symplectic filling of
$(Y, \xi)$ is symplectic deformation equivalent to a blow-up of a
PALF whose boundary is $\OB_\xi$. On the other hand, it is also
known that every weak symplectic filling of a rational homology
sphere can be modified into a strong symplectic filling \cite{oo}.
We conclude that any minimal symplectic  filling of $(L(p,q),
\xi_{can})$ admits a genus-zero PALF over $D^2$. In this section we
give an algorithm to describe any minimal symplectic filling of
$(L(p,q), \xi_{can})$ as an \emph{explicit} genus-zero PALF over
$D^2$.


\subsection{Lisca's classification of the fillings}
We first briefly review Lisca's classification \cite{l} of
symplectic fillings of $(L(p,q), \xi_{can})$, up to diffeomorphism.
Let
$$\frac{p}{p-q}=[b_1, \ldots, b_k],$$ where $b_i \geq 2$ for $1 \leq i \leq
k$.  A $k$-tuple of nonnegative integers $(n_1, \ldots, n_k)$ is
called admissible if each of the denominators in the continued
fraction $[n_1, \ldots, n_k]$ is positive. It is easy to see that an
admissible $k$-tuple of nonnegative integers is either $(0)$ or
consists only of positive integers. Let $\mathcal{Z}_{k} \subset
\mathbb{Z}^k$ denote the set of admissible $k$-tuples of nonnegative
integers $\textbf{n}=(n_1, \ldots, n_k)$ such that $[n_1, \ldots,
n_k] =0$, and let $$\mathcal{Z}_{k}(\textstyle{\frac{p}{p-q}}) = \{ (n_1,
\ldots, n_k)\in\mathcal{Z}_{k}\,|\, 0 \leq n_i \leq b_i \;
\mbox{for} \; i=1, \ldots , k\}.$$ Note that any $k$-tuple of
positive integers in $\mathcal{Z}_{k}$ can be obtained from $(1,1)$
by a sequence of strict blowups.

{\begin{Def} \label{blowup} A strict blowup of an $r$-tuple of
integers at the $j$th term is a map $\psi_j: \mathbb{Z}^r \to
\mathbb{Z}^{r+1}$ defined by
\begin{align*}
& (n_1, \ldots, n_{j}, n_{j+1}, \ldots, n_r) \mapsto  (n_1, \ldots,
n_{j-1}, n_{j}+1,1,n_{j+1}+1, n_{j+2}, \ldots,
n_r)
\end{align*}
for any $1 \leq j \leq r-1$ and by
\begin{align*} & (n_1, \ldots,
n_r) \mapsto  (n_1, \ldots, n_{r-1},
n_r+1,1)
\end{align*}
when $j=r$. The left inverse of a strict blowup at the $j$th term is called a strict blowdown at the $(j+1)$st term.
\end{Def}

Consider the chain of $k$ unknots in $S^3$ with framings $n_1, n_2,
\ldots, n_k$, respectively. For any $\textbf{n}=(n_1, \ldots, n_k)
\in \mathcal{Z}_{k}$, let $N(\textbf{n})$ denote the result of Dehn
surgery on this framed link. It is easy to see that $N(\textbf{n})$
is diffeomorphic to $S^1 \times S^2$. Let
$\textbf{L}=\bigcup_{i=1}^{k} L_i$ denote the framed link in
$N(\textbf{n})$, shown in Figure~\ref{han} in the complement of the
chain of $k$ unknots, where $L_i$ has $b_i -n_i$ components.

\begin{figure}[ht]
  \relabelbox \small {
  \centerline{\epsfbox{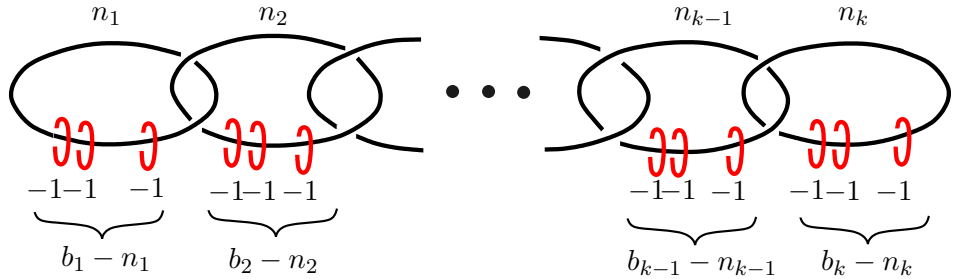}}}

\relabel{a}{$n_1$}

\relabel{b}{$n_2$}

\relabel{d}{$n_{k-1}$}

\relabel{e}{$n_k$}

\relabel{1}{$b_1-n_1$}

\relabel{2}{$b_2-n_2$}

\relabel{3}{$b_{k-1}-n_{k-1}$}

\relabel{4}{$b_k-n_k$}

\relabel{5}{$-1$}

\relabel{6}{$-1$}

\relabel{7}{$-1$}

\relabel{8}{$-1$}

\relabel{9}{$-1$}

\relabel{10}{$-1$}

\relabel{a1}{$-1$}

\relabel{a2}{$-1$}

\relabel{a3}{$-1$}

\relabel{a4}{$-1$}

\relabel{a5}{$-1$}

\relabel{a6}{$-1$}

\endrelabelbox
       \caption{Lisca's description of the filling $W_{(p,q)}(\textbf{n})$}

        \label{han}
\end{figure}

The $4$-manifold $W_{p,q}(\textbf{n})$ with boundary $L(p,q)$ is obtained by attaching
$2$-handles to $S^1 \times D^3$ along the framed link $\varphi(\textbf{L}) \subset S^1
\times S^2$ for some diffeomorphism $\varphi : N(\textbf{n})\to S^1 \times S^2$.  Note
that this description is a \emph{relative }handlebody decomposition of
 $W_{p,q}(\textbf{n})$ and it is independent of the choice of $\varphi$ since any self-diffeomorphism
of $S^1 \times S^2$ extends to $S^1 \times D^3$. According to Lisca,
any symplectic filling of ($L(p,q), \xi_{can})$ is
orientation-preserving diffeomorphic to a blowup of
$W_{p,q}(\textbf{n})$ for some $\textbf{n} \in
\mathcal{Z}_{k}(\frac{p}{p-q})$.

{\Rem\label{mo} In particular, for $p\neq4$, $(L(p,1), \xi_{can})$
has a unique minimal symplectic filling and, for $p \geq 2$,
$(L(p^2,p-1), \xi_{can})$ has two distinct minimal symplectic
fillings, up to diffeomorphism.}

\subsection{Another description of the fillings}\label{another}

Here we give another description of  $W_{p,q}(\textbf{n})$ which will lead to a
construction of a genus-zero PALF on this $4$-manifold with boundary.  First we slide
the unknot with framing $n_{k-1}$ over the unknot with framing $n_{k}$ and denote the
framing of the new unknot as $n'_{k-1}$. Next we slide the unknot with framing
$n_{k-2}$ over the unknot with framing $n'_{k-1}$ and proceed inductively until we
slide the unknot with framing $n_1$ over the one with framing $n'_2$ and let $n'_1$
denote its new framing.  By setting $n'_k=n_k$, the new framings of the surgery curves
are given by $n'_1, n'_2, \ldots, n'_k$, all of which can be computed inductively by
the standard formula for a handle-slide:
$$n'_i= n_i+n'_{i+1}  -2$$
for $1 \leq i \leq k-1$. Notice that these handle-slides are performed in the
complement of the link $\textbf{L}$ in Figure~\ref{han} and the result of Dehn surgery
on the new framed link is also diffeomorphic to $S^1 \times S^2$.

Moreover, this new surgery link can be viewed as the closure of a
braid in $S^3$. We order the strands of this braid using the
sub-indices of their associated framings. To visualize this braid,
imagine a trivial braid with $k$-strands, wrap the $k$th strand
$n'_k-1$ times around the first $k-1$ strands and then wrap the
strand indexed by $k-1$ around the first $k-2$ strands $n'_{k-1} -1
$ times and proceed inductively. See Figure~\ref{wrap} for an
illustration of ``wrapping around". To be more precise this braid is
given by

$$\prod_{j=2}^{j=k} (\s^{-1}_{j-1} \cdots \s^{-1}_1 \s^{-1}_1
\cdots \s^{-1}_{j-1})^{n'_j -1} $$ where $\s_1, \ldots, \s_{k-1}$
are the standard generators in the braid group with $k$ strands.

\begin{figure}[ht]
  \relabelbox \small {
  \centerline{\epsfbox{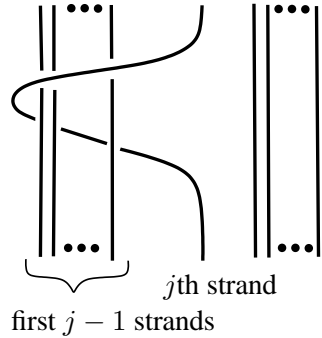}}}

\relabel{a}{first $j-1$ strands}

\relabel{b}{$j$th strand}

\endrelabelbox
       \caption{The $j$th strand wraps around the first $j-1$ strands once}
        \label{wrap}
\end{figure}

Each component $L_i$ of $\textbf{L}$ can now be viewed as an unknot linking the
first $i$ strands of this braid. As a result we get another relative handlebody
description of the $4$-manifold $W_{p,q}(\textbf{n})$, where the chain of unknots with
framings $n_1, \ldots, n_k$ in Lisca's description is replaced by unknots with framings
$n'_1, \ldots, n'_k$ braided as described above and the link $\textbf{L}$ plays the
same role in both descriptions.

\subsection{Open book decompositions of $S^1 \times
S^2$}\label{openbook}

We consider the open book decomposition compatible with the unique
tight contact structure on $S^1 \times S^2$ whose page is an annulus
and whose monodromy is the identity. We associate this open book to
the $1$-tuple $(0)\in \mathcal{Z}_{1}$. If $k>1$,
we stabilize this open book
once so that the new page is a disk with two holes and the new
monodromy is a right-handed Dehn twist around one of the holes. The
holes in the disk are ordered linearly from \emph{left} to
\emph{right} and the Dehn twist is around the second hole as shown
in Figure~\ref{stabil1}(a).

Depending on a blowup sequence from $(1,1)$ to $(n_1, \ldots, n_k)$,
we inductively stabilize this open book $k-2$ times as follows:
For the initial step corresponding to the blowup $(1,1) \to (2,1,2)$
we just split the second hole in Figure~\ref{stabil1}(a) into two
holes, so that both holes lie in the interior of the Dehn twist.
Then we relabel the holes as $1,2,3$ linearly from left to right and
add a stabilizing right-handed Dehn twist which encircles the holes
labelled as $1$ and $3$ as depicted in Figure~\ref{stabil1}(b). This
is certainly a positive stabilization, as one can attach a
$1$-handle in the interior of the second hole in
Figure~\ref{stabil1}(a), and let the stabilizing curve go over this
$1$-handle.

\begin{figure}[ht]
  \relabelbox \small {
  \centerline{\epsfbox{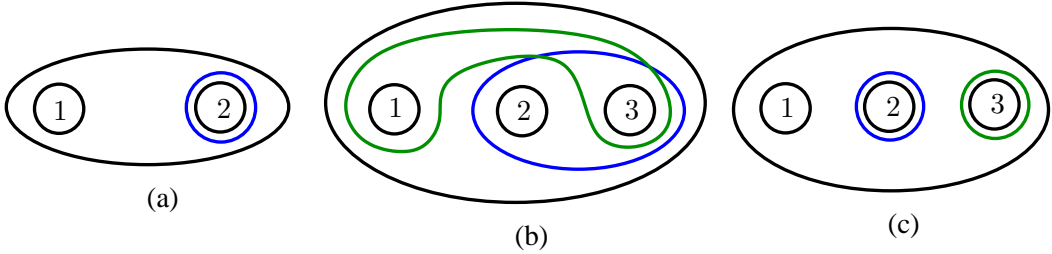}}}

\relabel{1}{$1$}

\relabel{2}{$2$}

\relabel{3}{$1$}

\relabel{4}{$2$}

\relabel{5}{$3$}

\relabel{6}{$1$}

\relabel{7}{$2$}

\relabel{8}{$3$}

\relabel{a}{(a)}

\relabel{b}{(b)}

\relabel{c}{(c)}

\endrelabelbox
       \caption{Positive stabilizations}
        \label{stabil1}
\end{figure}

Corresponding to the alternative blowup $(1,1) \to (1,2,1)$, we just insert a third
hole to the \emph{right} of the second hole so that this hole is \emph{not
included}---as opposed to the previous case---in the Dehn twist which already exists in
the initial open book. Then we add a stabilizing right-handed Dehn twist around this
new hole as shown in Figure~\ref{stabil1}(c).

Suppose that the page of the open book, corresponding to the result
of $r-2$ consecutive blowups starting from $(1,1)$,  is a disk
$D_r$  with $r$ holes (for $3 \leq r \leq k-1$) so that the
monodromy is the product of $r-1$ right-handed Dehn twists
$$x_1\cdots x_{r-1}.$$
Assume that the holes are ordered linearly from
\emph{left} to \emph{right} on the disk. If the next blowup occurs
at the $j$th term, for $1 \leq j \leq r-1$, then we insert a new
hole between the $j$th and $(j+1)$st holes (imagine splitting the
$(j+1)$st hole into two) and relabel the holes linearly from left to
right as $1,2, \ldots, r+1$. Let  $D_{r+1}$ denote the new disk with
$r+1$ holes and let $\widetilde{x}_i$ denote the right-handed Dehn
twist on $D_{r+1}$ induced from $x_i$. This means that if $x_i$
encircles the $(j+1)$st hole in $D_{r}$, then  $\widetilde{x}_i$
encircles the same holes as $x_i$ plus the new hole inserted to
obtain $D_{r+1}$, otherwise $x_i$ and $\widetilde{x}_i$ encircle the same holes.
To complete the stabilization,  we add a
right-handed Dehn twist along a curve $\b_j$ encircling the holes
labelled as $1,2,\ldots,j, j+2$, skipping the \emph{new hole} now
labelled as $j+1$ in $D_{r+1}$. As a result the monodromy of the new
open book is given by the product $$\widetilde{x}_1 \cdots
\widetilde{x}_{r-1} \b_j.$$

If, on the other hand, the next blowup occurs at the $r$th term, we
insert an $(r+1)$st hole to the right and add a stabilizing
right-handed Dehn twist $\a_{r+1}$ around this new hole labelled by
$r+1$. In this case, it is clear how to lift the Dehn twist $x_i$ in
$D_r$ to $\widetilde{x}_i$ in $D_{r+1}$ and the resulting monodromy
is $$\widetilde{x}_1 \cdots \widetilde{x}_{r-1} \a_{r+1}.$$

The page of the resulting open book decomposition of $S^1 \times
S^2$ corresponding to a strict blowup sequence from $(1,1)$ to the
positive  $k$-tuple $\textbf{n}=(n_1, \ldots, n_k)$ is a disk $D_k$ with
$k$ holes and the monodromy is given as the product of $k-1$
right-handed Dehn twists (ordered by the induction) along the
inserted stabilizing curves at each blowup. Note that if we think of the holes
in $D_k$ as being arranged counterclockwise in an annular neighbourhood
of the boundary, then each of the Dehn twists we consider is a \emph{convex}
Dehn twist.

The open book decomposition we have just constructed leads to yet another surgery
description of $S^1 \times S^2$. Take the closure of a trivial braid
with $k$ strands each of which has $0$-framing and insert
$(-1)$-framed surgery curves (ordered from top to bottom)
corresponding to the stabilizing curves linking this braid according
to the algorithm given above. By blowing down all the $(-1)$-surgery
curves we get  a \emph{framed braid} with $k$ strands whose closure
represents $S^1 \times S^2$.

\subsection{Equivalence of the two framed braids}\label{eq}

We claim that the framed braid with $k$ strands obtained by blowing down all the
$(-1)$-surgery curves in Section~\ref{openbook}, is exactly the same as the framed
braid obtained in Section~\ref{another} by handle-slides on the given chain of $k$
unknots. Our aim in this section is to prove this claim by induction.

First of all, we show that the framings of each strand with the same
index are equal in both braids. Let us use the notation $(n_1,
\ldots, n_r)'=(n'_1, \ldots n'_r)$ to denote the new framings of the
surgery curves after performing the handle-slides in
Section~\ref{another}. Then one can verify that the effect of a
blowup of an $r$-tuple at the $j$th term, for $1 \leq j \leq r-1$ is
given by
\begin{multline*}
 (n_1, \ldots, n_{j-1}, n_{j}+1,1,n_{j+1}+1, n_{j+2}, \ldots,
n_r)'\\
 =(n'_1+1, \ldots, n'_{j-1}+1, n'_{j}+1, n'_{j+1}, n'_{j+1}+1,
n'_{j+2}, \ldots,  n'_r ).
\end{multline*}

On the other hand,  for the induction step in the framed surgery
presentation described in Section~\ref{openbook}, we insert a zero
framed new strand between the $j$th and the $(j+1)$st strand and relabel the
strands linearly from left to right so that the new strand has index
$j+1$. We also insert a new $(-1)$-surgery curve linking the strands
$1,2,\ldots,j,j+2$ avoiding the new $(j+1)$st strand.
The induction hypothesis implies that by blowing down all the
$(-1)$-curves except the new one, the framings of the strands are
given by
$$(n'_1, \ldots,
n'_{j}, n'_{j+1}, n'_{j+1}, n'_{j+2}, \ldots,  n'_r).$$

We simply observe that blowing down the last inserted $(-1)$-surgery curve adds $1$ to
the new framing of each of the strands indexed by $1,2,\ldots,j,j+2$  which
is consistent with the blowup formula above.

Next we show that the two braids are in fact equivalent in the
complement of $\textbf{L}$. Suppose that the braids are equivalent
before we perform a blowup at the $j$th term. In the induction step
we insert a new strand between the $j$th and $(j+1)$st strand, which is
a parallel copy the $(j+1)$st strand in the braid described in
Section~\ref{another}. The induction hypothesis implies that by
blowing down all $(-1)$-curves except the new one, with the new
indexing, the $(j+1)$st strand links the $(j+2)$nd strand $n'_{j+1}$
times. They both wrap around the strands to the left of them
$n'_{j+1}-1$ times. The effect of blowing down the last inserted
$(-1)$-curve linking the strands $1,2,\ldots,j,j+2$ avoiding
the new strand (now indexed with $j+1$) is illustrated on the left
in Figure~\ref{equalbraid}, where the new strand is represented by
the thin curve.


\begin{figure}[ht]
  \relabelbox \small {
  \centerline{\epsfbox{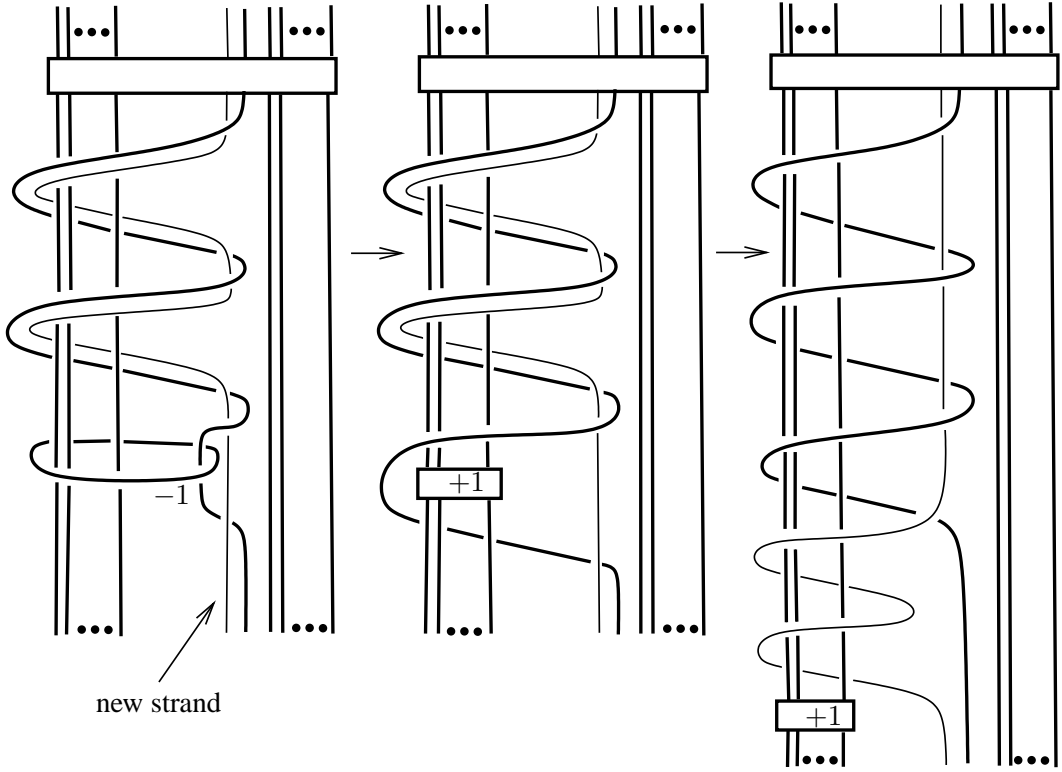}}}

\relabel{1}{$+1$}

\relabel{2}{$+1$}

\relabel{e}{$-1$}

\relabel{n}{new strand}

\endrelabelbox
       \caption{Blowing down the $(-1)$-curve}
        \label{equalbraid}
\end{figure}

By blowing down the last $(-1)$-curve, the strands
$1,2,\ldots,j,j+2$ will have a full right twist as shown in the middle in
Figure~\ref{equalbraid}. When we pull the ``spring" in the thin
curve down, it becomes clear how this $(j+1)$st strand wraps around
the strands to the left of it $n'_{j+1}-1$ times as depicted on the
right in Figure~\ref{equalbraid}. In this new braid the number of
times any strand wraps around the strands to the left of it is
consistent with the blowup formula given above. In particular, the
$(j+2)$nd strand wraps around the strands to the left of it
$n'_{j+1}$ times.

To verify our claim for the case of a blowup of an $r$-tuple at the
$j$th term for $j=r$ is much easier and it is left to the reader.

\subsection{Genus-zero PALF on the fillings}

The open book decomposition of  $S^1 \times S^2$ described in
Section~\ref{openbook}, corresponding to any sequence of strict
blowups from $(0)$ to a $k$-tuple $\textbf{n} \in
\mathcal{Z}_{k}(\frac{p}{p-q})$, is compatible with the unique tight
contact structure on $S^1 \times S^2$. The genus-zero PALF over
$D^2$ whose boundary is given by this open book is diffeomorphic to
$S^1 \times D^3$ since the tight contact $S^1 \times S^2$ has a
unique Stein filling up to diffeomorphism. A handlebody
decomposition of this PALF on $S^1 \times D^3$ can be obtained from
the closure of the framed braid in Section~\ref{openbook} by
converting the $0$-framed surgery curves---the strands of the
braid---to dotted circles representing $1$-handles, where  each
$(-1)$-surgery curve linking the strands of this braid represents a
vanishing cycle.

Inserting the link $\textbf{L}$ into this diagram completes the
handlebody decomposition of the desired PALF on
$W_{p,q}(\textbf{n})$, since each component of $\textbf{L}$ also
represents a vanishing cycle. This is because each component of
$\textbf{L}$ can be Legendrian realized on the planar page of the
open book of $S^1 \times S^2$.

As a consequence, the resulting contact structure on $L(p,q)$ is
obtained by Legendrian surgery from the standard tight contact $S^1
\times S^2$. The ordered vanishing cycles of this PALF on
$W_{p,q}(\textbf{n})$ can be explicitly described on a disk with $k$
holes by the algorithm given in Section~\ref{openbook}, where we add
a Dehn twist corresponding to each component of $\textbf{L}$ at the
end.  Summarizing we obtain

\begin{Thm} \label{pal} There is an algorithm to present any minimal symplectic filling of
the canonical contact structure on a lens space as an explicit
genus-zero PALF over the disk.
\end{Thm}
We would like to point out that the PALF in Theorem~\ref{pal} can be
obtained explicitly which therefore leads to an \emph{absolute}
handlebody decomposition of any symplectic filling at hand as
opposed to the relative decomposition depicted in Figure~\ref{han}.

\subsection{An example}
In the following we illustrate our algorithm to construct a
genus-zero PALF on the symplectic filling $W_{(81,47)}(\textbf{n})$
of the canonical contact structure on $L(81,47)$, where
$\textbf{n}=(3,2,1,3,2)$. Note that $\frac{81}{47}=[2,4,3,3,2]$ and
$\frac{81}{81-47}=[3,2,3,3,3]$.

According to Lisca's classification, $W_{(81,47)}(\textbf{n})$
represents one of the six distinct diffeomorphism classes of minimal
symplectic fillings of the canonical contact structure on
$L(81,47)$. The link $\textbf{L}$ in Lisca's description of the
filling in question has three components in total, two of which are
linking the third and one linking the fifth unknot in the chain
$\textbf{n}$ (see Figure~\ref{han2}).

\label{ex}\begin{figure}[ht!]
  \relabelbox \small {
  \centerline{\epsfbox{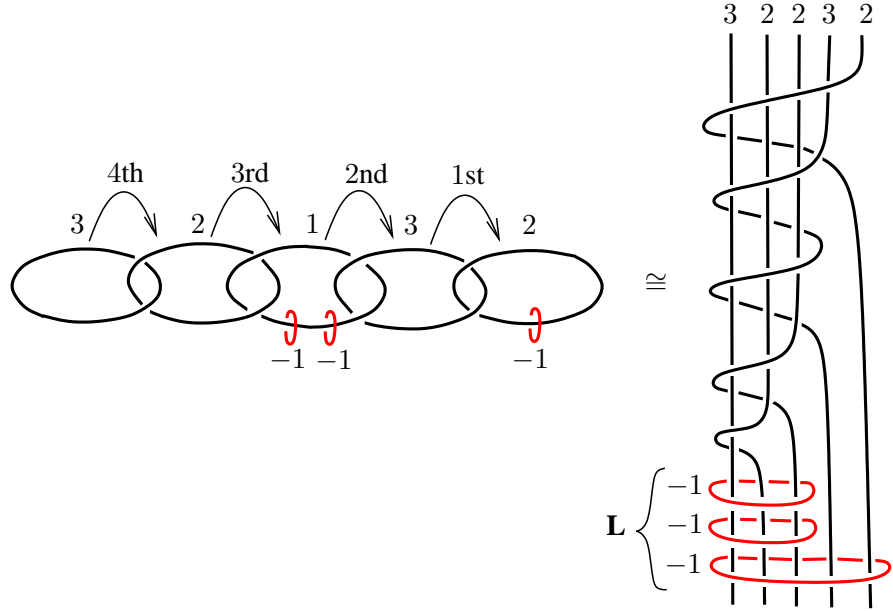}}}

\relabel{1}{$2$}

\relabel{2}{$3$}

\relabel{3}{$2$}

\relabel{4}{$2$}

\relabel{5}{$3$}

\relabel{a1}{$-1$}

\relabel{a2}{$-1$}

\relabel{a3}{$-1$}

\relabel{a}{$3$}

\relabel{b}{$2$}

\relabel{c}{$1$}

\relabel{d}{$3$}

\relabel{e}{$2$}

\relabel{f}{$-1$}

\relabel{g}{$-1$}

\relabel{h}{$-1$}

\relabel{l}{$\textbf{L}$ }

\relabel{y}{$\cong$}

\relabel{r1}{$1$st}

\relabel{r2}{$2$nd}

\relabel{r3}{$3$rd}

\relabel{r4}{$4$th}

\endrelabelbox
       \caption{Handle slides}

        \label{han2}
\end{figure}

First we slide $2$-handles in the chain over each other and obtain a
new surgery diagram as shown on the right in Figure~\ref{han2}. The
new unknots can be drawn as the closure of a braid and their
framings are given by $(n'_1, \ldots, n'_5)=(3,2,2,3,2)$. In
addition, two components of $\textbf{L}$ link the first three
strands, and one component links all the strands of this braid.

\begin{figure}[ht]
  \relabelbox \small {
  \centerline{\epsfbox{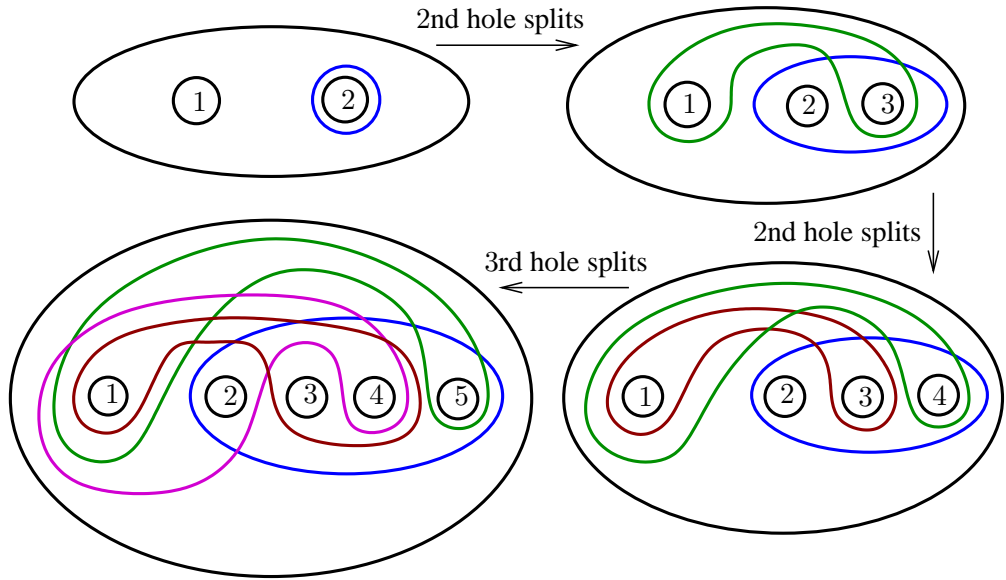}}}

\relabel{1}{$1$}

\relabel{2}{$2$}

\relabel{3}{$1$}

\relabel{4}{$2$}

\relabel{5}{$3$}

\relabel{6}{$1$}

\relabel{7}{$2$}

\relabel{8}{$3$}

\relabel{9}{$4$}

\relabel{a}{$1$}

\relabel{b}{$2$}

\relabel{c}{$3$}

\relabel{d}{$4$}

\relabel{e}{$5$}

\relabel{s1}{$2$nd hole splits}

\relabel{s2}{$2$nd hole splits}

\relabel{s3}{$3$rd hole splits}

\endrelabelbox
       \caption{Positive stabilizations of the standard open book of $S^1 \times S^2$}
        \label{open}
\end{figure}

\begin{figure}[ht]
  \relabelbox \small {
  \centerline{\epsfbox{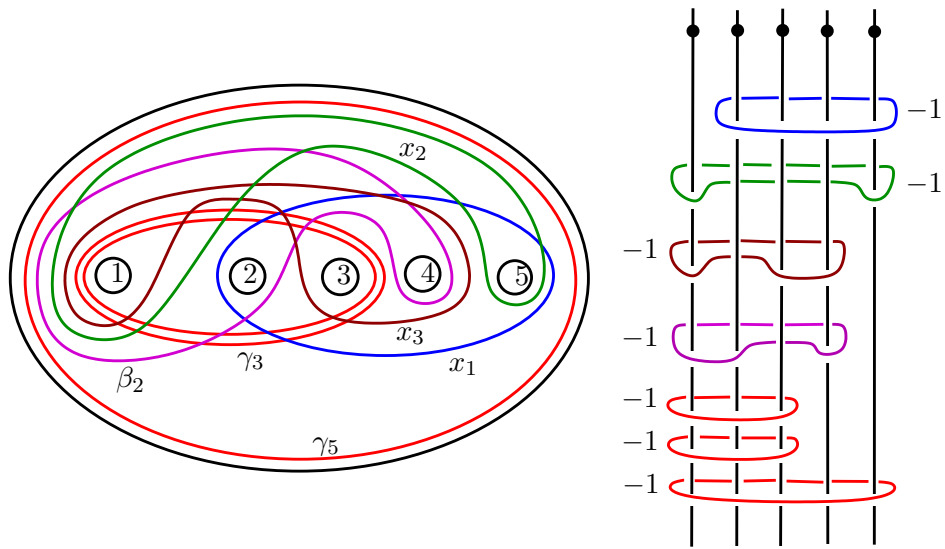}}}

\relabel{a}{$-1$}

\relabel{b}{$-1$}

\relabel{c}{$-1$}

\relabel{d}{$-1$}

\relabel{e}{$-1$}

\relabel{f}{$-1$}

\relabel{g}{$-1$}

\relabel{1}{$1$}

\relabel{2}{$2$}

\relabel{3}{$3$}

\relabel{k}{$4$}

\relabel{l}{$5$}

\relabel{x}{$x_1$}

\relabel{y}{$x_2$}

\relabel{z}{$x_3$}

\relabel{b2}{$\b_2$}

\relabel{g3}{$\g_3$}

\relabel{g5}{$\g_5$}

\endrelabelbox
       \caption{Monodromy $x_1x_2x_3\b_2\g^2_3\g_5$ of the PALF on $W_{(81,47)}((3,2,1,3,2))$ and its handlebody diagram}

        \label{palf}
\end{figure}

On the other hand, positive stabilizations of the standard open book
of $S^1 \times S^2$ corresponding to the blowup sequence $$(1,1) \to
(2,1,2) \to (3,1,2,2) \to (3,2,1,3,2)=\textbf{n}$$ is depicted in
Figure~\ref{open}. The monodromy of our PALF on
$W_{(81,47)}((3,2,1,3,2))$ is given as the product $$x_1x_2x_3
\b_2\g^2_3\g_5$$ of right-handed Dehn twists along the four
stabilizing curves $x_1, x_2, x_3,\b_2$ in the order they appear and
three more right-handed Dehn twists corresponding to the link
$\textbf{L}$ (see Figure~\ref{palf}). Two of these latter ones are
along two disjoint copies of a convex curve $\g_3$ encircling the
first three holes and one is along a convex curve $\g_5$ encircling
all the holes. Moreover, a handle decomposition of
$W_{(81,47)}((3,2,1,3,2))$ including five  $1$-handles, where one
can explicitly see the PALF is shown in Figure~\ref{palf}.


\section{Monodromy substitutions and rational blowdowns}\label{lant}

The lantern relation in the mapping class group of a sphere with
four holes was discovered by Dehn although Johnson named it as the
lantern relation after rediscovering it in \cite{jo}. This relation
and its generalizations have been effectively used recently in
solving some interesting problems in low-dimensional topology. The
key point is that the lantern relation (cf. Figure~\ref{lantern}) holds
in any subsurface of another surface which is homeomorphic to a
sphere with four holes.

\begin{figure}[ht]
  \relabelbox \small {
  \centerline{\epsfbox{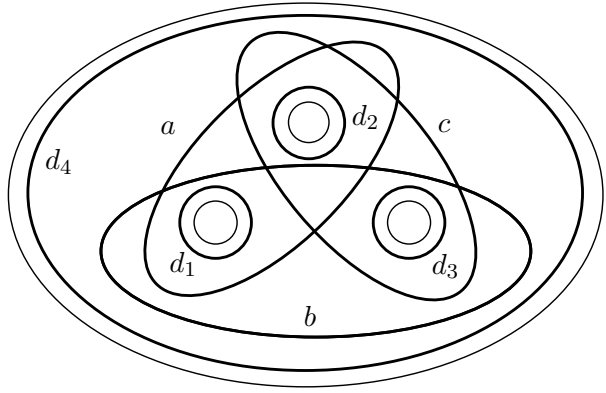}}}

\relabel{1}{$d_1$}

\relabel{2}{$d_2$}

\relabel{3}{$d_3$}

\relabel{4}{$d_4$}

\relabel{a}{$a$} \relabel{b}{$b$} \relabel{c}{$c$}

\endrelabelbox
       \caption{The lantern relation $d_1d_2d_3d_4=abc$
 on a four-holed sphere}
        \label{lantern}
\end{figure}

Suppose that there is a ``piece" in the monodromy factorization of a
(not necessarily positive or allowable) Lefschetz fibration which
appears as the left-hand side of the lantern relation. Deleting that
piece from the monodromy word and inserting the right-hand side is
called a lantern substitution. It was shown in \cite{eng} that the
effect of this substitution in the total space of the fibration is a
rational blowdown operation, which can be easily seen as follows:
The PALF with monodromy $d_1d_2d_3d_4$ is diffeomorphic to the $D^2$
bundle over $S^2$ with Euler number $-4$, while the PALF with
monodromy $abc$ is diffeomorphic to a rational $4$-ball with
boundary $L(4,1)$. Cutting a submanifold diffeomorphic to the
$D^2$-bundle over $S^2$ with Euler number $-4$ from a $4$-manifold
and gluing in a rational $4$-ball was named as a rational blowdown
operation by Fintushel and Stern \cite{fs}.

We would like to point out that the genus-zero PALF with monodromy
$d_1d_2d_3d_4$ and the genus-zero PALF with monodromy $abc$
represent the two distinct diffeomorphism classes of the minimal
symplectic fillings of $(L(4,1), \xi_{can})$.

Since the linear plumbing of $(p-1)$ disk bundles over $S^2$ with
Euler numbers $-(p+2)$, $-2$, $\ldots$, $-2$ has boundary $L(p^2,
p-1)$, which also bounds a rational $4$-ball,  the cut-and-paste
operation described above is defined similarly for this case
\cite{fs}. The corresponding monodromy substitution was discovered
and named as the daisy relation in \cite{emvhm}, which is
essentially obtained by repeated applications of the lantern
substitution. In fact, the PALFs given by the products of
right-handed Dehn twists appearing on the two sides of the daisy
relation represent the two distinct diffeomorphism classes of the
minimal symplectic fillings of $(L(p^2,p-1), \xi_{can})$ for any
$p\geq2$.

A generalization of Fintushel and Stern's rational blowdown
operation was introduced in \cite{p} involving the lens space
$L(p^2, pq-1)$ as the boundary.  The corresponding monodromy
substitution for this rational blowdown can be computed by the
technique introduced in \cite{emvhm}.

A \emph{rational blowdown along a linear plumbing graph} is the
replacement of a neighborhood of a configuration of spheres in a
smooth $4$-manifold which intersect according to a linear plumbing
graph whose boundary is $L(p^2, pq-1)$ by a rational $4$-ball with
the same oriented boundary.

\section{Symplectic fillings and rational blowdowns}

Our goal in this section is to prove our main result.

\begin{Thm} \label{ration} Any minimal symplectic filling of the canonical
contact structure on a lens space is obtained, up to diffeomorphism,
by a sequence of rational blowdowns along linear plumbing graphs
from the minimal resolution of the corresponding complex
two-dimensional cyclic quotient singularity.
\end{Thm}

{\Rem \label{sym} According to \cite{gm}, the rational blowdowns in
Theorem~\ref{ration} can be realized as \emph{symplectic} rational
blowdowns.} \\

It will be convenient to make the following definitions for the
proof of Theorem~\ref{ration}.

\begin{Def} \label{height} For a positive $k$-tuple
$\textbf{n}=(n_1,\ldots,n_k)\in\mathcal{Z}_{k}$, we say that $\textbf{n}$ has
\emph{height} $s$, and write
$\hgt(\textbf{n})=s$, if $s$ is the minimal number of strict blowups
required to obtain $\textbf{n}$ from an $l$-tuple of the form
$(1,2,\ldots,2,1)\in \mathbb{Z}^l$, which we will denote by
$\textbf{u}_l$, for $l\geq 2$. We set $\textbf{u}_1=(0)$ and define
$\hgt(\textbf{u}_1)=0$.
\end{Def}

It is easy to check that $$ \hgt(\textbf{n})=|\textbf{n}|-2(k-1), $$
for any  $\textbf{n}=(n_1,\ldots,n_k)\in\mathcal{Z}_{k}$, where
$|\textbf{n}|=n_1+\cdots+n_k$. \\

In addition, we slightly generalize the definition of the
$4$-manifold $W_{p,q}(\textbf{n})$ as follows:

\begin{Def} \label{wnm} For a pair of $k$-tuples
$\textbf{n}=(n_1,\ldots,n_k), \textbf{m}= (m_1,\ldots,m_k)\in
\mathbb{Z}^k$, with $\textbf{n}\in\mathcal{Z}_{k}$, we will denote by
$W(\textbf{n},\textbf m)$ the 4-manifold constructed as in
Section~\ref{lef} from the 3-manifold $N(\textbf{n}) \cong S^1
\times S^2$ and the framed link $\textbf L=\bigcup_{i=1}^k L_i$
associated to $\textbf m$, where $L_i$ consists of $|m_i|$
components as in Figure 1 with the components having framings $-1$
if $m_i>0$ and framings $+1$ if $m_i<0$.
\end{Def}

Note that if each $m_i\geq 0$ and $b_i:=n_i+m_i\geq 2$ for all $i$,
then there are unique integers $1\leq q<p$ with $(p,q)=1$ such that
$$\dfrac{p}{p-q}=[b_1,b_2\ldots, b_k].$$
In this case $W(\textbf{n},\textbf m)$ is just the minimal
symplectic filling $W_{p,q}(\textbf{n})$ of $L(p,q)$ given by Lisca.
Also note that if $\textbf m$ has precisely one component $m_j$
which is different from 0 with $m_j=\pm 1$ and $n_j=1$, then
W(\textbf{n},\textbf{m}) is a rational $4$-ball. To see this, note that
$H_1(W(\textbf{n},\textbf{m}),\bfq)$ and $H_2(W(\textbf{n},\textbf{m}),\bfq)$
are trivial precisely when the matrix describing the linking of the
attaching circles of the 2-handles with the dotted circles representing the 1-handles
is nondegenerate and it is easy to check that the latter holds when one imposes the above
conditions on $\textbf{m}$ and $\textbf{n}$.

By the algorithm in Section~\ref{lef}, the $4$-manifold
$W(\textbf{n},\textbf m)$ with boundary admits a genus-zero ALF
(\emph{achiral} Lefschetz fibration) over $D^2$. In other words, the
monodromy of the Lefschetz fibration will include left-handed Dehn
twists if $m_i <0$ for some $i$. In the following by the
\emph{monodromy factorization} of $W(\textbf{n}, \textbf{m})$ we
mean the monodromy factorization of this Lefschetz fibration over
$D^2$ (which may include some left-handed Dehn twists).
Moreover, by a \emph{cancelling pair of Dehn twists} we mean the
composition of a right-handed and a left-handed Dehn twist along two
parallel copies of some curve on a surface. Our proof of
Theorem~\ref{ration} is based on following preliminary result.

\begin{Lem} \label{intermed}
Given a pair of $k$-tuples $\textbf{n}=(n_1,\ldots,n_k), \textbf{m}=
(m_1,\ldots,m_k)\in \mathbb{Z}^k$, with
$\textbf{n}\in\mathcal{Z}_{k}$ and $s=\hgt(\textbf{n}) \geq 1$,
there exists a sequence of $k$-tuples $\textbf n_0,\ldots,
\textbf{n}_s\in\mathcal{Z}_{k}$ with $\textbf{n}_0 = \textbf{u}_k$
and $\textbf{n}_s=\textbf{n}$ such that, setting
$\textbf{m}_i=\textbf{n}+\textbf{m} - \textbf{n}_i$, the monodromy
factorization of $W(\textbf{n}_i,\textbf{m}_i)$ can be obtained from
the monodromy factorization of
$W(\textbf{n}_{i-1},\textbf{m}_{i-1})$ by a lantern substitution
together with, possibly, the introduction or removal of some
cancelling pairs of Dehn twists for $1\leq i\leq s$.
\end{Lem}

\begin{proof}
The proof will be by induction on $s$. Suppose that
$s=\hgt(\textbf{n})=1$. This means that $\textbf{n}=\psi_j(\textbf
u_{k-1})$ for some $1\leq j\leq k-2$, where
$\psi_j\colon\mathbb{Z}^{k-1}\to\mathbb{Z}^k$ denotes the strict
blowup at the $j$th term. Letting $\textbf m'=(m'_1,\ldots,m'_k)=
\textbf{n} + \textbf{m} - \textbf{u}_k$, we find that
\begin{equation*} m_i =
\begin{cases}
m'_i - 1 & \text{if $i=j$,} \\
m'_i + 1 & \text{if $i=j+1$,} \\
m'_i- 1 &  \text{if $i=j+2$,} \\
m'_i  & \text{otherwise}
\end{cases}
\end{equation*}
for any $\textbf{m}=(m_1,\ldots,m_k)\in\mathbb{Z}^k$. We compute the
monodromy factorizations $\phi$ and $\phi'$ of
$W(\textbf{n},\textbf{m})$ and $W(\textbf{u}_{k},\textbf{m}')$,
respectively. For this, consider a disk $D_k$ with $k$ holes ordered
linearly from left to right and label the boundary of the $i$th hole
$\a_i$, for $2\leq i\leq k$. Also, label the convex curve containing
the first $i$ holes $\g_i$, for $1\leq i\leq k$, and label the
convex curve containing the $(j+1)$st and the $(j+2)$nd holes
$\d_j$. Finally label the convex curve containing the first $j$
holes plus the $(j+2)$nd hole $\b_j$.  Here ``convex'' is used as in the sense
of Section~\ref{openbook}.
Following the algorithm given in the same section, we find that
$$ \phi'=\a_2\cdots\a_k \g^{m'_1}_1\cdots\g^{m'_k}_k$$
and
$$ \phi=\a_2\cdots\a_j\d_j\a_{j+3}\cdots\a_k\b_j
\g_1^{m'_1}\cdots \g_{j-1}^{m'_{j-1}} \g_{j}^{m'_j-1}
\g_{j+1}^{m'_{j+1}+1} \g_{j+2}^{m'_{j+2}-1}
\g_{j+3}^{m'_{j+3}}\cdots \g_k^{m'_k}. $$ We see that $\phi$ can
be obtained from $\phi'$ by the single lantern substitution
$$ \a_{j+1}\a_{j+2}\g_j\g_{j+2}=\d_j\b_j\g_{j+1}. $$
Note, however, that if either $m'_j \leq 0$ or $m'_{j+2} \leq 0$,
then we will need to introduce a cancelling pair of Dehn twists into
the monodromy factorization $\phi'$ before we can apply the lantern
substitution. Also, if $m'_{j+1} \leq -1$, then after applying the
lantern substitution we will remove a cancelling pair of Dehn twists
which appears in the monodromy. This finishes the proof for $s=1$ by
setting $\textbf{m}_0=\textbf{m}'$.

Now suppose that $t$ is a positive integer and it is known that for
every pair of $k$-tuples $\textbf{n}, \textbf{m}\in \mathbb{Z}^k$
with $\textbf{n}\in\mathcal{Z}_{k}$ and $s=\hgt(\textbf{n})\leq t$
there exists a sequence of $k$-tuples $\textbf n_0,\ldots,
\textbf{n}_s\in\mathcal{Z}_{k}$ with $\textbf{n}_0 = \textbf{u}_k$
and $\textbf{n}_s=\textbf{n}$ such that, setting
$\textbf{m}_i=\textbf{n}+\textbf{m} - \textbf{n}_i$, the monodromy
factorization of $W(\textbf{n}_i,\textbf{m}_i)$ can be obtained from
the monodromy factorization of
$W(\textbf{n}_{i-1},\textbf{m}_{i-1})$ by a lantern substitution
together with, possibly, the introduction or removal of some
cancelling pairs of Dehn twists for $1\leq i\leq s$. Let
$\textbf{n}, \textbf{m}\in \mathbb{Z}^k$ be a pair of $k$-tuples
with $\textbf{n}\in\mathcal{Z}_{k}$ and $s=\hgt(\textbf{n})=t+1$.
Then there is an $(k-1)$-tuple $\textbf{n}'\in\mathcal{Z}_{k-1}$
such that $\textbf{n}=\psi_j(\textbf{n}')$ and $\hgt(\textbf{n}')=
t$. Let $\rho_{j+1}\colon \bfz^k\to\bfz^{k-1}$ denote the map
$(l_1,\ldots,l_k)\mapsto (l_1,\ldots,\widehat{l}_{j+1},\ldots, l_k)$
given by omitting the $(j+1)$st entry. By the induction hypothesis,
there is a sequence of $(k-1)$-tuples
$\textbf{n}'_0,\ldots,\textbf{n}'_t\in\mathcal{Z}_{k-1}$ with
$\textbf{n}'_0=\textbf{u}_{k-1}$ and $\textbf{n}'_t=\textbf{n}'$
such that, setting $\textbf{m}'_i=\textbf{n}'+
\rho_{j+1}(\textbf{m}) -\textbf{n}'_i$, the monodromy factorization
of $W(\textbf{n}'_i,\textbf{m}'_i)$ can be obtained from the
monodromy factorization of $W(\textbf{n}'_{i-1},\textbf{m}'_{i-1})$
by a lantern substitution together with, possibly, the introduction
or removal of some cancelling pairs of Dehn twists for $1\leq i\leq
t$. Consider the sequence $\textbf{n}_i=\psi_j(\textbf{n}'_{i-1})$
for $1\leq i\leq s=t+1$ of $k$-tuples in $\mathcal{Z}_{k}$ obtained
by taking strict blowups at the $j$th term of the $(k-1)$-tuples in
the sequence $\textbf{n}'_0,\ldots, \textbf{n}'_t$ . Let
$\textbf{n}_0 = \textbf{u}_k$ and set
$\textbf{m}_i=\textbf{n}+\textbf{m}-\textbf{n}_i$ for $0\leq i\leq
s$. We claim that the monodromy factorization of
$W(\textbf{n}_i,\textbf{m}_i)$ can be obtained from the monodromy
factorization of $W(\textbf{n}_{i-1},\textbf{m}_{i-1})$ by a lantern
substitution together with, possibly, the introduction or removal of
some cancelling pairs of Dehn twists for $1\leq i\leq s$.

For $i=1$ the proof follows from above since $\hgt(\textbf{n}_1)=1$. Suppose that $i>1$.
Then the monodromy factorization $\phi'_{i-2}$ of $W(\textbf{n}'_{i-2},\textbf
m'_{i-2})$ has the form
$$ \phi'_{i-2} = c_1\cdots c_l, $$
where $c_r$ denotes a convex Dehn twists of $D_{k-1}$ for $1\leq
r\leq l$. It follows that the monodromy factorization $\phi_{i-1}$
of $W(\textbf{n}_{i-1}, \textbf{m}_{i-1})$ has the form
$$ \phi_{i-1} = \widetilde{c}_1\cdots \widetilde{c}_l\b_j\g_{j+1}^{m_{i-1,j+1}}, $$
where $\b_j$ and $\g_{j+1}$ are convex Dehn twist of $D_k$ as before
and $m_{i-1,j+1}$ denotes the $(j+1)$st component of
$\textbf{m}_{i-1}$. Here we have used the convention that if
$\sigma$ is a convex Dehn twist of $D_{k-1}$ around a collection of
holes $H$, then $\widetilde\sigma$ denotes the convex Dehn twist of
$D_k$ around the collection of holes $\widetilde{H}$ given by
$$ \widetilde{H} = \{r\,|\,1\leq r\leq j+1 \text{ and } r\in H\}\cup\{r+1\,|\,j+1\leq r\leq k-1
\text{ and } r\in H\}. $$
By the induction hypothesis, the monodromy factorization $\phi'_{i-1}$ of $W(\textbf{n}'_{i-1},
\textbf{m}'_{i-1})$ is obtained from the monodromy factorization $\phi'_{i-2}$
of $W(\textbf{n}'_{i-2},\textbf{m}'_{i-2})$ via a lantern relation of the form
$$ c_{i_1}c_{i_2}c_{i_3}c_{i_4}=c_{i_5}c_{i_6}c_{i_7}, $$
where, for each $r$, $c_{i_r}$ is a convex Dehn twist of $D_{k-1}$
which may or may not be included in the set of convex Dehn twists
$\{c_1,\ldots,c_l\}$, together with, possibly, the introduction or
removal of some cancelling pairs of Dehn twists. It follows easily
that the monodromy factorization $\phi_i$ of
$W(\textbf{n}_i,\textbf{m}_i)$ is obtained from the monodromy
factorization $\phi_{i-1}$ of $W(\textbf{n}_{i-1},\textbf m_{i-1})$
via a lantern relation of the form
$$ \widetilde{c}_{i_1}\widetilde{c}_{i_2}\widetilde{c}_{i_3}\widetilde{c}_{i_4}=
\widetilde{c}_{i_5}\widetilde{c}_{i_6}\widetilde{c}_{i_7}, $$
together with, possibly, the introduction or removal of some
cancelling pairs of Dehn twists, completing the proof of the
induction step and the lemma.
\end{proof}

\begin{proof}[Proof of Theorem~\ref{ration}]
Fix $1\leq q<p$ with $(p,q)=1$ and suppose that we are given a
minimal symplectic filling $W_{p,q}(\textbf{n})$ of $L(p,q)$, where
$\textbf{n}\in\mathcal{Z}_{k}(\frac{p}{p-q})$. Let
$\textbf{m}$ be the $k$-tuple of nonnegative integers corresponding
to the framed link $\textbf{L}$ so that $W_{p,q}(\textbf{n})=
W(\textbf{n},\textbf{m})$, and let $s=\hgt(\textbf{n})$. If $s=0$,
there is nothing to check. Suppose that $s\geq 1$ and consider the
sequence
\begin{equation} \label{blowdownseq}
  \textbf{n}=\textbf{n}^{0} \to \textbf{n}^{1} \to \cdots \to \textbf{n}^{s}
\end{equation}
given by taking the strict blowdown at the leftmost possible $1$. Here
$\textbf{n}^{i}\in\bfz^{k-i}$ for $0\leq i\leq s$. Observe that
$\textbf{n}^{s}=\textbf{u}_{k-s}$.
From the proof of Lemma~\ref{intermed}, there is an associated sequence
$\textbf{u}_{k}=\textbf{n}_0,\ldots,\textbf{n}_s=\textbf{n}$ such that, setting
$\textbf{m}_i=\textbf{n}+\textbf{m}-\textbf{n}_i$, the monodromy factorization of
$W(\textbf{n}_i,\textbf{m}_i)$ is obtained from the monodromy factorization of
$W(\textbf{n}_{i-1},\textbf{m}_{i-1})$ by a
lantern substitution together with, possibly, the introduction or
removal of some cancelling pairs of Dehn twists. Let $0=i_0<i_1<\cdots<i_r=s$
be the sequence of indices such that $\textbf m_i$ has all components nonnegative
if and only $i=i_j$ for some $j$. We claim that $W(\textbf{n}_{i_j},\textbf{m}_{i_j})$ is
obtained from $W(\textbf{n}_{i_{j-1}},\textbf{m}_{i_{j-1}})$ by a rational
blowdown for $1\leq j\leq r$. The proof is by induction on $r$.

Suppose that $r=1$, that is, $i_1=s$. We first show that $\textbf{n}=\textbf{n}_s$
contains exactly one component $n_j$ equal to $1$ with $1<j<k$. On the contrary, suppose
that $\textbf{n}$ contains at least two such components. Consider the strict blowdown
sequence in \eqref{blowdownseq} and let $\textbf{n}^{t}$ be the first tuple which has less
components equal
to $1$ than $\textbf{n}$. It follows from the assumption that $t<s$.
Let $\textbf{m}=\textbf{m}^{0},\ldots,\textbf{m}^{s}$
denote the associated sequence constructed as follows: if  $\textbf{n}^{i}$
is obtained from  $\textbf{n}^{i-1}$ by a strict blowdown at the $j$th term, let
$\textbf{m}^{i}=\rho_j(\textbf{m}^{i-1})$, where, as before, $\rho_j\colon\bfz^{k-i+1}\to\bfz^{k-i}$
is the map given by omitting the $j$th entry. For each pair $(\textbf{n}^i,
\textbf{m}^i)$, consider the sequence $(\textbf{n}_0^i=\textbf{u}_{k-i},\textbf{m}_0^i),
\ldots,(\textbf{n}_{s-i}^i=\textbf{n}^i,\textbf{m}_{s-i}^i=\textbf{m}^i)$ constructed
as in the proof of Lemma~\ref{intermed} from the portion of the blowdown sequence \eqref{blowdownseq}
beginning at $\textbf{n}^i$. Now consider the following diagram:
\begin{diagram}
(\textbf{n}=\textbf{n}_s^0,\textbf{m}=\textbf{m}_s^0) & \rDashto & (\textbf{n}^t=\textbf{n}_{s-t}^{t},\textbf{m}_{s-t}^{t}) & \rDashto & (\textbf{n}_0^{s}=\textbf{u}_{k-s},\textbf{m}_0^{s}) \\
\uDashto                                              &          &  \uDashto                                                &          & \\
(\textbf{n}_t^0,\textbf{m}_t^0)                       & \rDashto & (\textbf{n}_{0}^{t}=\textbf{u}_{k-t},\textbf{m}_{0}^{t}) &          & \\
\uDashto                                              &          &                                                          &          & \\
(\textbf{n}_{0}^0=\textbf{u}_{k},\textbf{m}_{0}^0)    &          &                                                          &          & \\
\end{diagram}

\noindent
Note that, by construction, every component of $\textbf{n}_{s-t}^t+\textbf{m}_{s-t}^t$ is
greater than $1$. (Here the condition that each strict blowdown is taken at the leftmost possible
$1$ is essential.) Hence every component of $\textbf{n}_{0}^t+\textbf{m}_{0}^t$ is greater
than $1$.
Since $\textbf{n}_{0}^t=\textbf{u}_{k-t}=(1,2,\ldots,2,1)$, it follows that every
component of $\textbf{m}_0^t$ is nonnegative. Now note that $\textbf{m}_{l+1}^{i-1}$ can be
obtained from $\textbf{m}_l^i$ as follows: suppose that $\textbf{n}_{s-i}^i$ is obtained from
$\textbf{n}_{s-i+1}^{i-1}$ by strictly blowing down at the $j$th term (and hence that
$\textbf{n}_{l}^i$ is obtained from $\textbf{n}_{l+1}^{i-1}$ also by strictly blowing down at
the $j$th term for $0\leq l\leq s-i$.), then
$\textbf{m}_{l+1}^{i-1}=\chi_j(\textbf{m}_l^i)$, where $\chi_j\colon\bfz^{k-i}\to\bfz^{k-i+1}$
is the map $(z_1,\ldots,z_{k-i})\mapsto (z_1,\ldots,z_{j-1},m_j,z_j,\ldots,z_{k-i})$
given by splicing into the $j$th position the $j$th component of $\textbf{m}$. It follows that
every component of $\textbf{m}_t^0$ is nonnegative contradicting the fact that $r=1$.
This proves that $\textbf{n}$ contains exactly one component $n_j$ equal to $1$.

We now proceed as follows: Given $\textbf{n}$, suppose that $n_j$ is the only component
that is equal to $1$, with $1<j<k$. Let $\textbf{m}'=(0,\ldots,0,1,0,\ldots,0)$, where the $1$ is
in the
$j$th position. Then $\textbf{m}'\leq\textbf{m}$, since every component of $\textbf{n}+\textbf{m}$,
except possibly the last,
is greater than $1$ and $\textbf{m}$ is nonnegative. Now note that $W(\textbf{u}_k,\textbf{m}'_0)$
can be rationally blown down to $W(\textbf{n},\textbf{m}')$, where
$\textbf{m}'_0=\textbf{n}+\textbf{m}'-\textbf{u}_k$ (since $W(\textbf{n},\textbf{m}')$ is a rational
$4$-ball).
By replacing the ``piece'' of the monodromy factorization of $W(\textbf{u}_k,\textbf{m}_0)$
that corresponds to the monodromy factorization of $W(\textbf{u}_k,\textbf{m}'_0)$,
where $\textbf{m}_0=\textbf{n}+\textbf{m}-\textbf{u}_k$,
by the monodromy factorization of $W(\textbf{n},\textbf{m}')$, we see that $W(\textbf{n},\textbf{m})$
is obtained from $W(\textbf{u}_k,\textbf{m}_0)$ by a rational blowdown.

Now assume that $l\geq 1$ and the claim is known to hold whenever $r\leq l$. Suppose that
$r=l+1$ and consider the following diagram:
\begin{diagram}
(\textbf{n}=\textbf{n}_{i_r}^0,\textbf{m}=\textbf{m}_{i_r}^0) & \rDashto & (\textbf{n}_{i_r-i_1}^{i_1},\textbf{m}_{i_r-i_1}^{i_1})        & \rDashto & (\textbf{n}_{0}^{i_r}=\textbf{u}_{k-i_r},\textbf{m}_{0}^{i_r}) \\
\uDashto                                                      &          &  \uDashto                                                      &          & \\
(\textbf{n}_{i_1}^0,\textbf{m}_{i_1}^0)                       & \rDashto & (\textbf{n}_{0}^{i_1}=\textbf{u}_{k-i_1},\textbf{m}_{0}^{i_1}) &          & \\
\uDashto                                                      &          &                                                                &          & \\
(\textbf{n}_{0}^0=\textbf{u}_{k},\textbf{m}_{0}^0)            &          &                                                                &          & \\
\end{diagram}

\noindent
By the previous step, we know that there is exactly one $j$ with $1<j<k$ such that the $j$th component
of $\textbf{n}_{i_1}^0$ is $1$. It follows that
$W(\textbf{n}_{i_1}^0,\textbf{m}_{i_1}^0)$ is obtained from $W(\textbf{n}_{0}^0,\textbf{m}_{0}^0)$
by a rational blowdown. Thus it is sufficient to show that $W(\textbf{n}_{i_r}^0,\textbf{m}_{i_r}^0)$
is obtained
from $W(\textbf{n}_{i_1}^0,\textbf{m}_{i_1}^0)$ by a sequence of rational blowdowns. For this, consider
the pair $(\textbf{n}_{i_r-i_1}^{i_1},\textbf{m}_{i_r-i_1}^{i_1})$. Since in the sequence
$\textbf{m}_0^{i_1},\textbf{m}_1^{i_1},\ldots,\textbf{m}_{s-i_1}^{i_1}$ the only tuples with
all components nonnegative are precisely the ones with subindices $0<i_2-i_1<\cdots<i_r-i_1=s-i_1$,
it follows
from the induction hypothesis that $W(\textbf{n}_{i_r-i_1}^{i_1},\textbf{m}_{i_r-i_1}^{i_1})$ is
obtained
from $W(\textbf{u}_{k-i_1},\textbf{m}_{0}^{i_1})$ by a sequence of rational blowdowns. Now, arguing
as before we find
that $W(\textbf{n}_{i_r}^0,\textbf{m}_{i_r}^0)$ is obtained from
$W(\textbf{n}_{i_1}^0,\textbf{m}_{i_1}^0)$
by a sequence of rational blowdowns completing the induction step and the proof of the theorem.
\end{proof}

The content of Corollary 5.2 and Theorem 6.1 in \cite{l} can be
recovered as a corollary:

{\Cor \label{ste} Any minimal symplectic filling of the canonical
contact structure on a lens space can be realized as a Stein
filling, .i.e. the underlying smooth $4$-manifold with boundary
admits a Stein structure whose induced contact structure on the
boundary agrees with the canonical one.}

\begin{proof} Any minimal symplectic  filling of the canonical contact
structure on a lens space admits a PALF over $D^2$ by
Theorem~\ref{pal} (also by \cite[Theorem1]{w}). This implies that
the underlying smooth $4$-manifold with boundary admits a Stein
structure whose induced contact structure on the boundary is
compatible with the open book induced from the PALF \cite{ao}. By
the proof of Theorem~\ref{ration}, the induced open book on the
boundary is fixed for all distinct PALFs constructed for a given
lens space. The desired result follows since we know that the
induced open book on the boundary of the canonical PALF on the
minimal resolution is compatible with the canonical contact
structure \cite{oz}.
\end{proof}

{\Cor The canonical contact structure on a lens space admits a
unique minimal symplectic filling---represented by the Stein
structure via the PALF we constructed on the minimal resolution---up
to symplectic rational blowdown and symplectic deformation
equivalence.}

\begin{proof}
This result follows from the combination of Theorem~\ref{ration},
Remark~\ref{sym}, Corollary~\ref{ste} and the fact that each
diffeomorphism type of a minimal symplectic filling of the canonical
contact structure on a lens space carries a unique symplectic
structure up to symplectic deformation equivalence which fills the
contact structure in question \cite{bs}.
\end{proof}

\subsection{An example}

We would like to describe how one can obtain the symplectic filling
$W_{(81,47)}((3,2,1,3,2))$ from the minimal resolution
 $W_{(81,47)}((1,2,2,2,1))$ by a single rational
 blowdown.

The monodromy of the the canonical PALF on
 $W_{(81,47)}((1,2,2,2,1))$, which is illustrated in  Figure~\ref{5holes-lant1}(a), can be expressed as $$\a_2\a_3\a_4a_5\g^2_1 \g_3
\g_4 \g^2_5$$ by our algorithm using the blowup sequence $$(1,1) \to
(1,2,1) \to (1,2,2,1) \to (1,2,2,2,1).$$

In the following we describe a sequence of lantern substitutions,
together with introduction or removal of some cancelling pairs of
Dehn twists, to obtain the PALF (see Figure~\ref{palf}) we
constructed on the symplectic filling $W_{(81,47)}((3,2,1,3,2))$
from
 the canonical PALF (see Figure~\ref{5holes-lant1}(a)) on the minimal resolution
 $W_{(81,47)}((1,2,2,2,1))$.

\begin{figure}[ht]
  \relabelbox \small {
  \centerline{\epsfbox{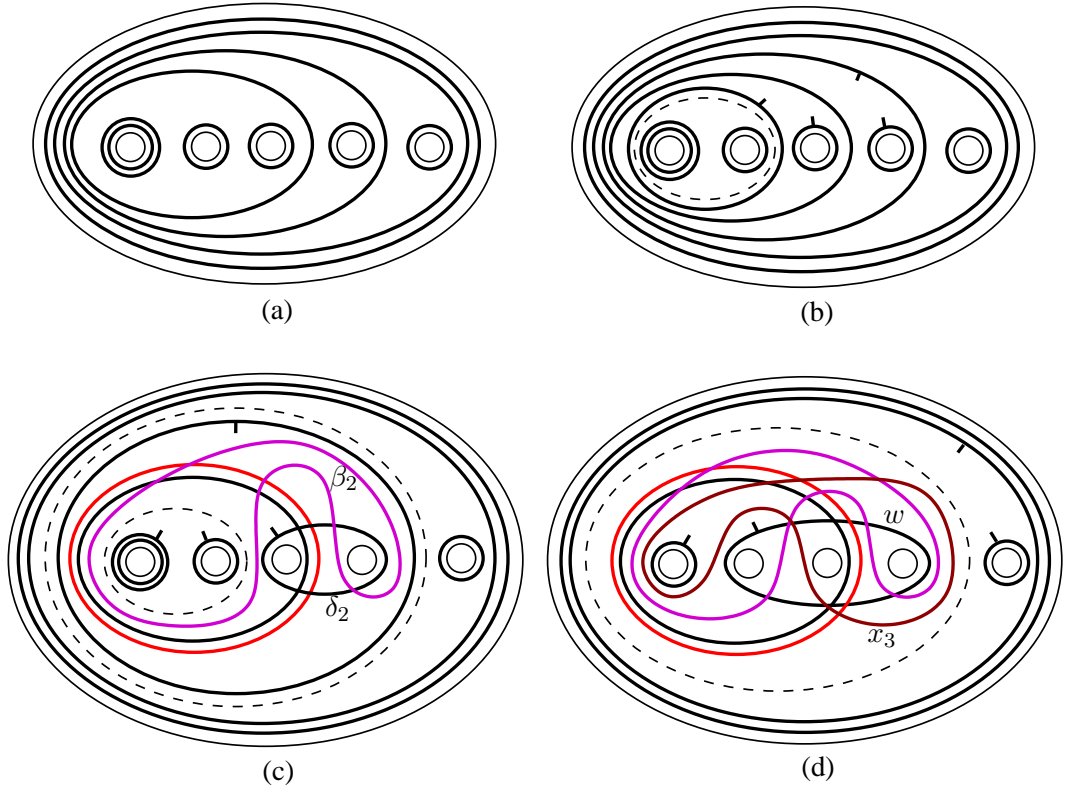}}}

\relabel{a}{(a)} \relabel{b}{(b)} \relabel{c}{(c)} \relabel{d}{(d)}

\relabel{x}{$w$} \relabel{y}{$x_3$}

\relabel{d2}{$\d_2$} \relabel{b2}{$\b_2$}

\endrelabelbox
       \caption{Thicker curves indicate right-handed Dehn twists on a $5$-holed disk, where left-handed Dehn
twists are drawn as dashed curves}
        \label{5holes-lant1}
\end{figure}

We first insert a cancelling pair of Dehn twists along two parallel
copies of a curve encircling the first two holes to obtain the ALF
in Figure~\ref{5holes-lant1}(b) with monodromy
$$\a_2\a_3\a_4a_5 \g^2_1 (\g^{-1}_2  \g_2)\g_3 \g_4 \g^2_5.$$

We apply a lantern substitution $\g_2\a_3\a_4\g_4= \d_2\b_2\g_3$ as
indicated in Figure~\ref{5holes-lant1}(b), to obtain the new ALF
depicted in Figure~\ref{5holes-lant1}(c) with monodromy
$$\a_2\a_5\g^2_1 \g^{-1}_2 (\g_4 \g^{-1}_4) \d_2\b_2\g_3 \g_3\g^2_5$$ where we
also inserted a pair of cancelling Dehn twists along two parallel
copies of a curve encircling the first four holes.

Next we apply a second lantern substitution $\g_1 \a_2 \d_2  \g_4=
\g_2wx_3$ indicated in Figure~\ref{5holes-lant1}(c), to obtain the
new ALF depicted in Figure~\ref{5holes-lant1}(d) with monodromy
$$\a_5\g^{-1}_4\g_1 (\g^{-1}_2 \g_2) wx_3 \b_2\g^2_3\g^2_5=\a_5
\g^{-1}_4\g_1 wx_3 \b_2\g^2_3\g^2_5$$  where we removed a pair of
cancelling Dehn twists encircling the first two holes. A final
lantern substitution $\g_1 w \a_5 \g_5 = \g_4x_1x_2$ is applied as
indicated in Figure~\ref{5holes-lant1}(d), together with the removal
of a pair of cancelling Dehn twists encircling the first four holes,
to obtain a PALF whose monodromy is $$(\g^{-1}_4\g_4) x_1 x_2 x_3
\b_2\g^2_3\g_5=  x_1 x_2 x_3 \b_2\g^2_3\g_5.$$

\noindent It is clear that this monodromy is equivalent to the
monodromy of the PALF on the symplectic filling
$W_{(81,47)}((3,2,1,3,2))$
 depicted in
Figure~\ref{palf}.

Using the notation in Lemma~\ref{intermed}, the above sequence of
three lantern substitutions can be expressed as
\begin{align*}
W_{(81,47)}((1,2,2,2,1)) &=
W((1,2,2,2,1), (2,0,0,1,1,2))\\
&\phantom{=}\to W((1,3,1,3,1), (2,-1,2,0,2))\\
&\phantom{=}\to W((2,2,1,4,1), (1,0,2,-1,2))\\
&\phantom{=}\to W((3,2,1,3,2), (0,0,2,0,1))
= W_{(81,47)}((3,2,1,3,2)).
\end{align*}

We show that the filling $W_{(81,47)}((3,2,1,3,2))$ is in fact
obtained from the minimal resolution $W_{(81,47)}((1,2,2,2,1))$ by a
single rational blowdown as follows: The monodromy of the PALF on
$W_{(81,47)}((3,2,1,3,2))$ can be obtained from the monodromy of the
PALF on $W_{(81,47)}((1,2,2,2,1))$ by a single monodromy
substitution (see Figure~\ref{open3}) as
$$\underline{\a_2\a_3\a_4a_5\g^2_1 \g_4 \g_5 } \g_3 \g_5 =
\underline{x_1x_2x_3 \b_2\g_3 }\g_3 \g_5, $$ which is the
combination of the three lantern substitutions together with the
introduction or removal of cancelling pairs of Dehn twists.

\begin{figure}[ht]
  \relabelbox \small {
  \centerline{\epsfbox{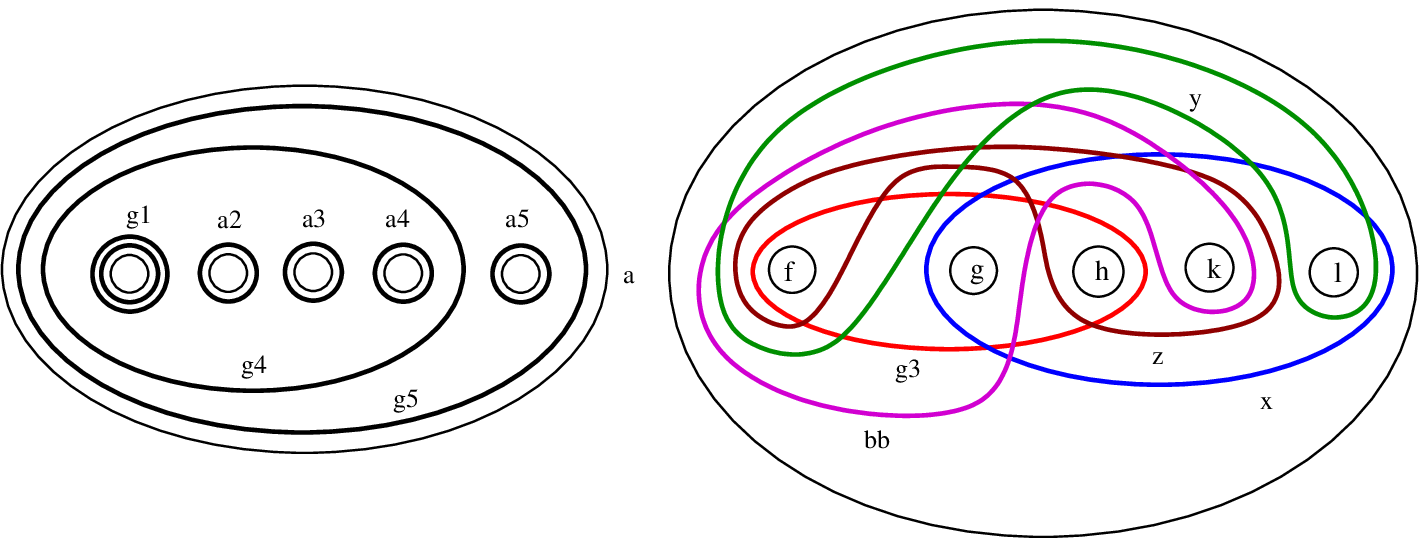}}}

\relabel{a}{$=$}

\relabel{y}{$x_2$} \relabel{z}{$x_3$} \relabel{x}{$x_1$}

\relabel{bb}{$\b_2$} \relabel{g3}{$\g_3$}

\relabel{g4}{$\g_4$} \relabel{g5}{$\g_5$}

\relabel{g1}{$\g_1$} \relabel{a2}{$\a_2$}

 \relabel{a3}{$\a_3$}   \relabel{a4}{$\a_4$}
 \relabel{a5}{$\a_5$}

\endrelabelbox
       \caption{A monodromy substitution: $\a_2 \a_3 \a_4\a_5\g^2_1 \g_4  \g_5  = x_1 x_2 x_3  \b_2\g_3.$}
        \label{open3}
\end{figure}

The PALF represented on the left-hand side in Figure~\ref{open3} is
diffeomorphic to the linear plumbing of disk bundles over $S^2$ with
Euler numbers $-2,-5,-3$, which can be directly checked by drawing
the handlebody diagram of this PALF and applying some handle slides
and cancellations. On the other hand, the PALF on the right-hand
side is a rational homology $4$-ball since the curves in the
monodromy spans the rational homology of the genus-zero fiber. We
conclude that this monodromy substitution corresponds to a rational
blowdown since
$$[-2,-5,-3]=-\dfrac{5^2}{5.3-1}.$$

\begin{Rem} When we run our algorithm for the two distinct minimal
symplectic fillings of $(L(p^2, p-1), \xi_{can})$, for any $p \geq
2$,  we obtain another proof of the daisy relation \cite{emvhm}. Our
method would yield many more interesting ``positive" relations in
the mapping class groups of planar surfaces.
\end{Rem}

We would like to finish with the following question: Does
Theorem~\ref{ration} hold true for minimal symplectic fillings of
any Milnor fillable contact $3$-manifold supported by  a
\emph{planar} open book?



\begin{thebibliography}{99999}

\bibitem{ao} S. Akbulut and B. Ozbagci, {\em Lefschetz fibrations on compact
    Stein surfaces}, Geom. Topol. {\bf
    5} (2001), 319--334.


\bibitem{bono}  M. Bhupal and K. Ono, {\em Symplectic fillings of links of
    quotient surface singularities},
    Nogoya Math. J. 207 (2012), 1--45.


\bibitem{bs}  M. Bhupal and A. I. Stipsicz, \emph{Smoothings of
    singularities and symplectic topology},
    Deformations of Surface Singularities, Bolyai Soc. Math. Stud., Vol. {\bf 23}, 2013.


\bibitem{bd} F. A. Bogomolov and B.  de Oliveira, \emph{Stein small
    deformations of strictly pseudoconvex
    surfaces}, Birational algebraic geometry (Baltimore, MD, 1996),  25--41, Contemp. Math., 207, Amer. Math.
    Soc., Providence, RI, 1997.



 \bibitem{cnp}
C. Caubel, A. N\'{e}methi, and P. Popescu-Pampu, \emph{Milnor open books and Milnor
fillable contact $3$-manifolds},  Topology  45 (2006),  no. 3, 673--689.








\bibitem{eng} H. Endo and Y. Z. Gurtas, {\em Lantern relations and rational
    blowdowns}, Proc. Amer. Math. Soc.
    138 (2010), no. 3, 1131--1142.

\bibitem{emvhm} H. Endo, T. E. Mark,  and J. Van Horn-Morris, {\em Monodromy
substitutions and rational blowdowns}, J. Topol.  4  (2011),  no. 1,
227-253.





\bibitem{fs} R. Fintushel and R. J. Stern, {\em Rational blowdowns of
    4-manifolds}, J. Diff. Geom. 46 (1997),
    181--235.


\bibitem{gm} D. Gay and T. E. Mark, \emph{Convex plumbings and Lefschetz
fibrations}, to appear in Journal of Symplectic Geometry, Volume 11,
Number 3.








\bibitem{gs} R. E. Gompf and A. I. Stipsicz, \emph{$4$-manifolds and Kirby
    calculus}, Graduate Studies in
    Mathematics, 20. American Mathematical Society, Providence, RI, 1999.





\bibitem{jo} D. Johnson, {\em Homeomorphisms of a surface which act trivially
    on homology}, Proc. Amer. Math.
    Soc.  78  (1980), no. 1, 135--138.





\bibitem{l} P. Lisca, {\em On symplectic fillings of lens spaces}, Trans.
    Amer. Math. Soc. 360 (2008), no. 2,
    765--799.





\bibitem{oo} H. Ohta and K. Ono, {\em Simple singularities and topology of symplectically
filling 4–-manifold}, Comment. Math. Helv. 74 (1999) 575--590.



\bibitem{oz} B. Ozbagci, {\em Surgery diagrams for horizontal contact
structures}, Acta Math. Hungar.  120 (2008),  no. 1-2, 193--208.


\bibitem{ozst} B. Ozbagci and A. I. Stipsicz, {\em Surgery on contact
    3-manifolds and Stein surfaces}, Bolyai
    Soc. Math. Stud., Vol. {\bf 13}, Springer, 2004.



\bibitem{p} J. Park, {\em Seiberg-Witten invariants of generalised rational
blow-downs}, Bull. Austral. Math. Soc.  56  (1997),  no. 3,
363--384.

\bibitem{sc} S. Sch\"{o}nenberger, {\em Determining symplectic fillings from planar
open books}, J. Symplectic Geom.  5  (2007),  no. 1, 19--41.


\bibitem{w} C. Wendl, {\em Strongly fillable contact manifolds and J -holomorphic
foliations}, Duke Math. J.  151  (2010),  no. 3, 337--384.



\end{thebibliography}
\end{document}